\documentclass{article}
\usepackage{graphicx}
\usepackage{amsfonts}
\usepackage{epstopdf}
\usepackage{amsmath}
\usepackage{amsthm}
\usepackage{xcolor}
\usepackage{epstopdf}
 \usepackage{graphics}
\usepackage{color}
\usepackage{amssymb}

\newcommand{\diag}{\operatorname{diag}}
\newcommand{\const}{\operatorname{const}}

\newcommand{\Real}{\mathbb{R}}
\newcommand{\Comp}{\mathbb{C}}

\newcommand{\eps}{\varepsilon}

\newcommand{\set}[1]{\left\{#1\right\}}
\newcommand{\seq}[1]{\left<#1\right>}
\newcommand{\norm}[1]{\left\Vert#1\right\Vert}  

\usepackage{color}

\newtheorem{theorem}{Theorem}
\newtheorem{corollary}[theorem]{Corollary}
\newtheorem{lemma}[theorem]{Lemma} 
\newtheorem{proposition}[theorem]{Proposition}

\theoremstyle{definition}
\newtheorem{remark}[theorem]{Remark}

\title{Matrix  methods for Pad\'e approximation: numerical calculation of  poles, zeros and residues.}

\author{Luca Perotti\footnotemark[1]~ and Micha\l{} Wojtylak\footnotemark[2]~\footnotemark[3]}

\begin{document}

\date{\today}
\maketitle
\begin{abstract}

A  representation of the Pad\'e approximation of the $Z$-transform of a signal as a resolvent of a tridiagonal matrix $J_n$ is given.  Several  formulas for the poles, zeros and residues of the Pad\'e approximation in terms of the matrix $J_n$ are proposed. Their numerical stability is tested and compared. Methods for computing forward and backward errors are presented. 
\end{abstract}

\noindent


\renewcommand{\thefootnote}{\fnsymbol{footnote}}
\footnotetext[1]{Department of Physics, Texas Southern University, Houston, Texas 77004 USA \texttt{perottil@tsu.edu}}
\footnotetext[2]{Instytut Matematyki, Wydzia\l{} Matematyki i Informatyki,
Uniwersytet Jagiello\'nski, Krak\'ow, ul. \L ojasiewicza 6, 30-348 Krak\'ow, Poland
   \texttt{michal.wojtylak@uj.edu.pl}.}
\footnotetext[3]{   Supported by the Alexander von Humboldt Foundation with
  a  by a Return Home Scholarschip.}

\renewcommand{\thefootnote}{\arabic{footnote}}


\section{Introduction}

When trying to extract the information part from a discrete finite data sequence $s=s_0,s_1,\dots,s_{(2n-1)} $ ($s_i\in\Comp$) where the signal is immersed in uncorrelated or weakly correlated additive noise, it is sometimes convenient to consider its $Z$--transform, which is a function of a complex variable defined by the series 
\begin{equation}\label{oscilform1}
Z(w)=\sum_{k=0}^\infty s_k w^{-k},
\end{equation}
convergent at least outside the unit circle. 

As the exact calculation of $Z(w)$ is impossible due to the finiteness of the  experimental sequence, one considers a Pad\'e approximant of $Z(w)$, usually --for reasons that we shall present in Section \ref{choice}-- the $[n-1/n]$ one, i.e. a function of the form ${P_n(w)}/{Q_n(w)}$,
where $P_n(w)$ and $Q_n(w)$ are polynomials of degrees respectively at most $n-1$ and $n$, without common zeros and such that
$$
Z(w)- \frac{P_n(w)}{Q_n(w)}=\mathcal{O}(w^{-n-1}),\quad w\to\infty.
$$

As our starting point is the technique presented in  \cite{bessis2009universal}, we shall start with a brief review of its salient points. 
Consider a discrete finite data series $s=s_0,s_1,\dots,s_{(2n-1)} $ ($s_i\in\Comp$) where each term is the sum of a signal part and a noise part: $s_i=s_i^{\text{info}}+s_i^{\text{noise}}$.
The task is to find the signal part, knowing only $s$ and  making minimal theoretical assumptions on both $s^{\text{info}}$ and $s^{\text{noise}}$.
We assume that the signal part is a finite sum of damped oscillations
\begin{equation}\label{oscilform}
 s_k^{\text{info}}=\sum_{p=1}^{P}A_p\mathrm{e}^{\mathrm{i}\gamma_p k},\quad k=0,1,2,\dots,\quad \gamma_p=\frac{\omega_p T}{N}
 \end{equation}
 where $P$ is the number of oscillations, $N=2n$ is the length of the data sequence, $T$ is the time interval,  $A_p$ is the amplitude,
 $\omega_p=2\pi f_p+i b_p$,
  $f_p$ is the frequency, $ b_p>0$ is the damping factor for $p=1,2,\dots P.$
 The noise is usually a complex uncorrelated (white, uniform or Gaussian) or lightly correlated ($1/f$ pink or ARMA) stationary noise \cite{bessis2009universal}, though other kinds of noise can be considered, as for example noise with a Cauchy distribution \cite{bessis2013noise}.  

The method described in \cite{bessis2009universal} is based on the following matrix representation 
$$
\frac{P_n(w)}{Q_n(w)}=\seq{(wI_{n} -  J_{n})^{-1} \mathbf{e}_0,\mathbf{e}_0}.
$$
Here $J_{n}$ is a special $n\times n$ tridiagonal matrix, constructed recursively from the data sequence,  $\mathbf{e}_0$ is the first vector of the canonical basis and $I_{n}$ is the identity matrix.  In Section \ref{matrixform} we will give a detailed construction of the matrix $J_{n}$. Consequently, the eigenvalues of $J_n$ are the zeros of $Q_n(w)$, and the eigenvalues of the matrix $J_n$ with the first row and first column removed are the zeros of $P_n(w)$.

%

\textit{In this way, the problem of signal decomposition is transformed into a problem of describing and computing the spectrum of a large, tridiagonal matrix.}

We note here that, should the series (\ref{oscilform1}) be truncated at $k=2n$ instead than at $k=2n-1$, the Pad\'e approximant would be of the $[n/n]$ type. Its construction is very similar to that of the $[n-1/n]$ one; when needed, we shall --throughout the paper-- point out the differences.

Our paper is related also to the theory of Hankel matrices and pencils, especially \cite{boley1998vandermonde,bryc2006spectral,tyrtyshnikov1994bad}. To clarify it, we remark that the polynomial $Q_n(w)$ is given by the determinant of a Hankel pencil:
$$
Q_n(w)=\det(U_0 - w U_1), 
$$
where $U_0,U_1\in\Comp^{n+1}$ are Hankel matrices defined by
$$
U_0[i,j] = s_{i+j+1},\quad U_1[i,j]=s_{i+j},\quad i,j=0,\dots, n,
$$
see \cite{barone2008new}. 
Hence, the matrix $J_n$ is a companion matrix of the Hankel pencil $U_0 - z U_1$, in the sense that if $U_0 - z U_1$ has distinct simple eigenvalues, then the spectrum of $J_n$  coincides with the spectrum of $U_0 - z U_1$, see Corollary \ref{cor:linalg} for a precise statement. Although this fact pertains to pure linear algebra, its proof relies heavily on the Pad\'e approximation theory.  

 Our paper is organized as follows: Section \ref{choice}, meant for a general readership, outlines the physical and data analysis motivations for our research; Section \ref{gen} gives a brief introduction to Pad\'{e} approximants; in Section \ref{pz} we derive the algorithms for the construction of the $J_n$ matrices for computing poles and zeros of the diagonal and subdiagonal Pad\'{e} approximants to the Z-function; Section \ref{residues} presents several formulas for the residues of the poles; finally,  Section \ref{num} compares the numerical results of the different methods presented in the previous two sections.
\section{Motivations}\label{choice}

\subsection{Comparison with FFT}

 While the method we outlined is more time consuming than the standard Fast Fourier Transform (FFT), it has some advantages that warrant its use:

\begin{itemize}
\item[a.] The structure of the sequence $s$ is encoded in the zeros and poles of ${P_n(w)}/{Q_n(w)}$ in a way that permits to distinguish signal from noise poles, even when the noise is high. For information of the form \eqref{oscilform} we have:

\medskip

\begin{itemize}
\setlength{\itemsep}{1pt}
  \setlength{\parskip}{5pt}
  \setlength{\parsep}{0pt}
\item[(i)] There are $P<n$ poles of ${P_n(w)}/{Q_n(w)}$, each of them
corresponding to a signal frequency whose amplitude can be constant or varying exponentially. The knowledge of the poles $z_1,\dots z_P$ and their residues completely determines all the signal parameters. These poles are stable: they do not move significantly when the length of the data sequence is changed (only a  slight movement is caused by contamination from the noise poles due to the nonlinearity of the Pad\'e approximant).
\item[(ii)] The remaining $n-P$ poles correspond to the noise and converge with $n\to\infty$ to the roots of unity. These poles are not stable: they move when the length of the data sequence is changed. Furthermore, each of these zeros is located close to a zero of $P_n(w)$, forming a so-called ``Froissart-doublet" \cite{froissart1969approximation}. 
\end{itemize}

\medskip
\item[b.] Provided the signal is not too week and the sequence is long enough (see below), signal poles repel both noise poles and zeros and therefore stand out among the other poles \cite{perotti2012IEEE}. 
\item[c.] The zeros of $Q_n(w)$ are not bound to a lattice on the unitary circle as is the case for the FFT, but are free to move in the complex plane. This has two positive consequences: 

\medskip

\begin{itemize}
\setlength{\itemsep}{1pt}
  \setlength{\parskip}{5pt}
  \setlength{\parsep}{0pt}
\item[(i)] The so called ``super-resolution": the Nyquist-Shannon sampling theorem \cite{sha} --that limits the frequency resolution to $1/T$, where $T$ is the total sampling time-- is verified only in average, but not locally, as poles can cluster where the density of signals is high. Reductions of a factor $100$ with respect to FFT of the number of data points required to resolve signals with similar frequencies have been observed \cite{barone1996super,barone1998some,perotti2013enhanced}.
\item[(ii)]  For FFT damped signals appear as peaks whose half--width equals the decay factor.  Hence,  weak signals next to strong damped ones can be covered by the wide peak of the FFT of the strong signal; also, nearby damped signals may not be resolved as they appear as a single wide peak. The current method 
avoids these problems, as damped signals appear as poles off the unit circle.
\end{itemize}

\medskip

\end{itemize}

\subsection{Applications}

Pad\'{e} analysis has found application in several areas, particularly when the system oscillations are damped (because of point c.(ii) above); or when they are chirped, or have a complicated time dependence. In these latter cases the signal is usually analyzed by looking at the signal frequencies in sliding ``time windows" and plotting the frequencies found in each window vs. the window starting time; ``super-resolution" (point c.(i) above) allows the use of shorter ``windows" than FFT and to thus follow the changes in time of the system frequencies avoiding the often complicated sidebands that appear when using the longer ``windows" required to obtain a sufficient resolution with FFT. 

 For applications involving damped signals, we can mention magnetic resonance spectroscopy \cite{belkic2006vivo}; for damped chirped signals in high noise, gravitational wave bursts \cite{perotti2014bursts}; and for signals with frequencies oscillating in a complicated way, syncronized neuronal hyppocampal rythms.
 
Other possible fields of application due to the presence of damped signals include financial modeling \cite{junior2011shocks,wu2012damped} and acoustic localization of underground oil fields. Applications where strong noise is present (and therefore properties a. and b. are useful) include measuring-while-drilling data transmission in oil industry, early detection of structural faults in mechanical systems such as helicopter shafts, and detection of water distribution pipe leaks \cite{azen08,leaks}. 

\subsection{The context}

Several other methods of calculating poles and residues of Pad\'e approximations are known, starting from the Jacobi one \cite{jacobi1846} (which we shall use in Theorem \ref{simplezeros}) to the Hankel matrices linear pencil one \cite{barone2008new,barone2010}. However, we do not know of any other direct way in the literature of calculating the zeros, the usual procedure being to find $P(w)$ by taking the first $n$ terms of the product $Z(w)Q(w)$, which can be numerically unstable. The method presented in Section \ref{pz} is based on matrix calculations and has also theoretical consequences: see e.g. Theorem \ref{pert} below.


While analytic continuation of the $Z$--transform inside the unit circle is a trivial matter in absence of noise, noise creates a natural boundary on the unit circle itself \cite{steinhaus1930wahrscheinlichkeit}. The boundary is usually assumed to be of the Carleman class \cite{carleman1926fonctions}, but a proof is still missing \cite{bessis2013noise}.
 M. Froissart \cite{froissart1969approximation} has shown that, in the general noisy case, the natural boundary generated by the noise \cite{steinhaus1930wahrscheinlichkeit} is approximated by doublets of poles and zeros of the Pad\'e approximants (Froissart doublets) surrounding the vicinity of the unit circle, with the mean distance of the members of each doublet proportional to the noise magnitude. This phenomenon was later discussed in the context of noise filtering in \cite{Bessis1996Pade,fournier1993universal,fournier1995statistical}. While the relationship between the elongation $|z_j^n - \lambda_j^n|$ of a Froissart doublet and the residue of the corresponding noise pole is clear, the rate of convergence of the elongation of the Froissart doublets to zero is still unknown.

This fact, although a formal argument is missing, suggest that  for the a noisy signal case the $[n/n]$ approximant should be optimal.
However, the $[n-1/n]$ approximant is much easier in computation and allows more comfortable theoretical background, as will be seen in the course of the manuscript.

 A further reason for the present study is the ongoing exploration of the statistical properties of noise poles, zeros, and residues of the Pad\'e approximant of $Z(w)$ \cite{barone2013universality,bessis2009universal,bessis2013noise}, whose aim is to better distinguish noise and signal so as to improve the current denoising methods. The maximum data sequence length we were able to analyze on a P.C. before encountering memory problems was $4\times 10^5$. The results of this study will be presented in a forthcoming paper.

\section{General Pad\'e approximation theory}\label{gen}

 We shall  briefly discuss in this section the well known relation between the Pad\'e approximation and continuous fractions in the special case of interest. Our aim is to review the existing theory in a way which is tailored for subsequent application. This will be done in two theorems, whose proofs can be found e.g. in \cite{Wall}.

Consider a power series
\begin{equation}
G(z) = \sum_{j=0}^{\infty}s_jz^j, \label{szere}
\end{equation}
convergent in some disc $D\left(0,r\right)$, $r>0$. By its 
 $[L/M]$ \textit{ Pad\'e approximation } we understand a rational function 
 $$
 G_{[L/M]}(z)=\frac{P(z)}{Q(z)},
 $$
 where $P(z)$ and $Q(z)$ are polynomials without common zeros, of degrees at most  $L,M$ respectively and such that  $Q(0)\neq 0$ and
 \begin{equation}
 G(z)- \frac{P(z)}{Q(z)} = O(z^{L+M+1})  \qquad (z\to 0). \label{duzeo}
 \end{equation}
   Further, for a continuous fraction 
\begin{equation} \label{S}
\cfrac{ a_0}{ b_1+\cfrac{z a_1}{ b_2+\cfrac{z a_3}{ b_3+{\atop{\ddots}}}}}, \quad a_i,b_i\in\Comp\setminus\set0
\end{equation}
we denote by $S_k(z)$ the $k$-th partial sum of the fraction, and let 
$$
S_k(z)=\frac{A_k(z)}{B_k(z)}
$$
where the polynomials $A_k(z)$ and $B_k(z)$ do not have common factors. The following theorem lists the basic properties of the polynomials $A_k$ and $B_k$. 

\begin{theorem}\label{generalPade1} Let the polynomials $A_k$ and $B_k$ be as above, then the following statements hold.
\begin{itemize}
\item[(i)] The polynomials $A_k$ and $B_k$ satisfy the following  Frobenius recurrence relations
\begin{equation} \label{14}
 \begin{gathered} 
   zA_{-2}(z)=1, \qquad A_{-1}(z)=0,\\
 A_{k+2}(z)= b_{k+3}A_{k+1}(z)+z a_{k+2}A_k(z),\qquad k\geq -2 ,
\end{gathered}
 \end{equation} 
 and
 \begin{equation}\label{15}
 \begin{gathered}
    B_{-2}(z)=0, \qquad B_{-1}(z)=1,\\
  B_{k+2}(z)= b_{k+3}B_{k+1}(z)+z a_{k+2}B_{k}(z), \qquad k\geq -2 , 
  \end{gathered}
   \end{equation}
\item[(ii)] The polynomial $A_k(z)$ is of degree $\lfloor\frac{k}{2}\rfloor$ and the polynomial $B_k$ is of degree $\lfloor\frac{k+1}{2}\rfloor$ for $k\geq 0 $.
\item[(iii)] The polynomials $A_k(z)$ and $B_k(z)$ have no common zeros for $k\geq 0$.

\item[(iv)] $B_k(0)\neq 0$ for $k\geq 0 $.

\item[(v)] In the expansion of the partial sum in a Taylor series 
$$
\frac{A_k(z)}{B_k(z)}=\sum_{i=0}^{k-1}s_iz^i+\mathcal{O}(z^k)\quad  \quad (z\to 0),\quad k\geq 0
$$
(which is possible due to (iv))
 the coefficients $s_i$ $(i\geq 0)$ do not depend on $k$.
\item[(vi)] If the power series 
$$
\sum_{i=0}^\infty s_iz^i 
$$
converges in some neighborhood of the origin, then $A_{2n-1}(z)/B_{2n-1}(z)$ is its $[n-1/n]$ Pad\'e approximation and $A_{2n}(z)/B_{2n}(z)$ is its $[n/n]$ Pad\'e approximation.
\end{itemize}
\end{theorem}

In view of the theorem above, to compute the diagonal or sub-diagonal Pad\'e approximation of a given power series one needs to provide formulas for 
the coefficients $a_i$, $b_i$. These coefficients are non-unique; those we present here --while not optimal-- have the advantage of being fraction-less, which is necessary in numerical applications using integer  multiple precision mathematics  to stabilize the procedure when noise is absent or very weak. We refer the reader to \cite{beckermann2000fraction} for a general construction of fraction-less Pad\'e approximants and a proof that the coefficients $a_i$ and $b_i$ below are indeed fraction-less.

\begin{theorem}\label{generalPade2} Let $s_0,s_1,\ldots\in\Comp$ be given and such that the power series
$$
G(z)=\sum_{k=0}^\infty s_k z^k
$$
is convergent in some neighborhood of zero. 
 Let us choose
  \begin{equation}\label{15b}
 a_l=(-1)^l\frac{f_l^0}{f_{l-2}^0},  \quad  b_{l+1}=(-1)^l\frac{f_{l-1}^0}{f_{l-2}^0},   \quad l\geq 0  
  \end{equation}
where the numbers $f_{k}^l$ are defined recursively as 
 \begin{equation}\label{999}
  \begin{gathered}
 f_{-2}^k= f_{-1}^k=\delta_{0,k},\quad f_0^k=s_k,\quad k\geq 0\\
  f_{l+1}^k=\frac{(-1)^l}{f_{l-2}^0}\left[f_l^0f_{l-1}^{k+1}-f_{l-1}^0f_l^{k+1}\right],\quad l\geq 0 .
  \end{gathered}
  \end{equation}
  Then, the continuous fraction \eqref{S} expands as a power series $G(z)$ in the sense of statement (v) of Theorem \ref{generalPade1}.

  \end{theorem}

\begin{proof} (see \cite{beckermann2000fraction}) 
Let us define 
\begin{equation}\label{f's}
   f_l(z)=\sum_{i=0}^\infty f_l^i z^i, \qquad l\geq -2 .
   \end{equation} 
Then $f_{-2}(z)=f_{-1}(z)=1$ and  $f_0(z)=G(z)$. 
Observe, that because of \eqref{f's}
\begin{equation}
   b_{l+1}f_l^0= a_lf_{l-1}^0,\quad l\geq 0.
  \end{equation}
  Consequently
 \begin{equation}
   b_{l+1}f_l(z)+zf_{l+1}(z)= a_lf_{l-1}(z),\quad l\geq 0 ,  \end{equation}
which  implies that all the sequences in \eqref{f's} are convergent in some neighborhood of the origin. 
Furthermore, $f_l(0)\neq 0$ ($l=-1,0,\dots$). 
Define
\begin{equation}
 u_l(z)=\frac{f_l(z)}{f_{l-1}(z)}, \qquad l\geq 0  .\label{4}
 \end{equation}
The functions $u_l(z)$ satisfy the following recurrence relation 
  \begin{equation}
  u_l(z) = \frac{ a_l}{ b_{l+1}+zu_{l+1}(z)}, \qquad l \geq  0. \label{generate}
  \end{equation}   
As $u_0(z)=G(z)$ the proof is complete. 
 
  \end{proof}
  
  Summarizing  Theorems \ref{generalPade1} and \ref{generalPade2} we get the following result.
  
 \begin{corollary}  \label{generalPade3} 
Let $s_0,s_1,\dots\in\Comp$ be given and such that the power series
$$
G(z)=\sum_{k=0}^\infty s_k z^k
$$
is convergent in some neighborhood of zero. Define the numbers $f_{l}^k$ as in \eqref{999}, $a_i$, $b_i$ as in \eqref{15b} and the polynomials $A_k(z)$ and $B_k(z)$ as in \eqref{14} and \eqref{15}, respectively.
Then $A_{2n-1}(z)/B_{2n-1}(z)$ and $A_{2n}(z)/B_{2n}(z)$ are respectively the $[n-1/n]$ and $[n/n]$ Pad\'e approximation of $G(z)$. Furthermore, to compute 
$A_{2n-1}(z)/B_{2n-1}(z)$ and $A_{2n}/B_{2n}$ one needs only the entries $s_0,\dots,s_{2n-1}$ and  $s_0,\dots,s_{2n}$, respectively.
 \end{corollary}
  
  \begin{proof}	
  Only the last statement requires a justification. Observe that by \eqref{14} and \eqref{15} to compute $A_{2n-1}, B_{2n-1}$ one needs to compute $a_0,\dots a_{2n-1}$ and $b_1,\dots b_{2n}$.  For this aim one needs to know $f_{0}^0,\dots ,f_{2n-1}^0$, see \eqref{15b}. By \eqref{999} it is enough to know $f_0^0,\dots f_0^{2n-1}$, i.e. $s_0,\dots s_{2n-1}$. Similarly, to compute $A_{2n}$ and $B_{2n}$ one requires $a_0,\dots a_{2n}$ and $b_1,\dots b_{2n+1}$ and therefore $f_{0}^0,\dots ,f_{2n}^0$ and $s_0,\dots s_{2n}$
  \end{proof}
  
  \section{Poles and zeros}\label{pz}

 So far we have constructed fraction-less coefficients $a_i$, $b_i$, which can be computed  using exact integer arithmetic. 
However, these coefficients grow rapidly with $l\to\infty$, see \cite{beckermann2000fraction} for an explanation, which makes their use impossible for large $n$. Furthermore, our aim is to compute poles, zeros and residues of the Pad\'e approximants, for which we have to resign exact arithmetics anyway. In the present section we shall derive numerically stable algorithms for computing the poles, zeros and residues.

  \subsection{Towards a numerical algorithm}\label{Towards}
  
   Our first step is to normalize the polynomials $A_n(z)$ and $B_n(z)$ such that 
  $\widehat A_n(0)=\widehat B_n(0)=1$. 
Basing on (\ref{14}) and  (\ref{15}) we easily get that
 \begin{equation}
   A_l(0)= a_0 b_2 b_3\dots b_{l+1}, \qquad B_l(0)= b_1 b_2 b_3\dots b_{l+1}, \qquad l\geq 0.
 \end{equation} 
Hence,
 \begin{equation}\label{17}
 \begin{split}
 \widehat{A_l}(z)&:=\frac{A_l(z)}{ a_0 b_2 b_3\dots b_{l+1}} 
 = \frac{A_l(z)}{(-1)^{\frac{l(l+1)}{2}}\frac{f_0^0}{f_{-2}^0}\frac{f_{0}^0}{f_{-1}^0}\frac{f_1^0}{f_0^0}\cdots\frac{f_{l-1}^0}{f_{l-2}^0}} =\\&= 
 \frac{A_l(z)}{(-1)^{\frac{l(l+1)}{2}}f_0^0f_{l-1}^0}=\frac{A_l(z)}{(-1)^{\frac{l(l+1)}{2}}s_0f_{l-1}^0},
 \end{split}
 \end{equation}
 \begin{equation}\label{18}
 \begin{split}
 \widehat{B_l}(z)&:=\frac{B_l(z)}{ b_1 b_2 b_3\dots b_{l+1}}=\frac{B_l(z)}{(-1)^{\frac{l(l+1)}{2}}\frac{f_{-1}^0}{f_{-2}^0}\frac{f_0^0}{f_{-1}^0}\frac{f_1^0}{f_0^0}\cdots\frac{f_{l-1}^0}{f_{l-2}^{0}}}=\\&=\frac{B_l(z)}{(-1)^{\frac{l(l+1)}{2}}f_{l-1}^0}.
 \end{split}
 \end{equation} 
 Hence,
  \begin{equation}\label{finalPades}
    \frac{A_{2n-1}(z)}{B_{2n-1}(z)} = \frac{s_0\widehat{A}_{2n-1}(z)}{\widehat{B}_{2n-1}(z)}\quad \text{ and } \quad\frac{A_{2n}(z)}{B_{2n}(z)} = \frac{s_0\widehat{A}_{2n}(z)}{\widehat{B}_{2n}(z)}
    \end{equation}
are respectively the $[n-1/n]$ and $[n/n]$ Pad\'e approximation of $G(z)$.  We now provide recurrence relations for even and odd polynomials separately.
 \begin{lemma}\label{2rec}
 The the following recurrence relations hold
      \begin{equation} \label{rekurencjaDlaA}
            \begin{gathered}
              \widehat{A}_{-1}(z)=0, \qquad \widehat{A}_{1}(z)=1,\\
             \widehat{ A}_{0}(z)= 1   ,\quad    \widehat{ A}_{2}(z)=1+zr_2 \\
               \widehat{A}_{l+1}(z)=[1+(r_{l}+r_{l+1})z]\widehat{A}_{l-1}(z)-z^2r_{l-1}r_{2l}\widehat{A}_{l-3}(z),\quad l\geq 2 ,          \end{gathered}
               \end{equation}         
      \begin{equation}\label{rekurencjaDlaB}
            \begin{gathered}
                         \widehat{B}_{-1}(z)=1, \qquad \widehat{B}_{1}(z)=1+r_1z, \\ 
                         \widehat{B}_{0}(z)= 1   ,\quad    \widehat{B}_{2}(z)= 1+(r_1+r_2)z\\
            \widehat{B}_{l+1}(z)=[1+(r_{l}+r_{l+1})z]\widehat{B}_{l-1}(z)-z^2r_{l-1}r_{2l}\widehat{B}_{l-3}(z), \quad l\geq 2   ,
              \end{gathered}
            \end{equation}
 where           
$$                    
r_0=0,\quad r_l=(-1)^{l-1}\frac{f_l^0f_{l-3}^0}{f_{l-2}^0f_{l-1}^0},\quad l\geq 1 .
$$          
 \end{lemma}
 
\begin{proof} 
 
 Using \eqref{14} and \eqref{17} we get
 
  \begin{equation}
 \begin{gathered}
  a_0 b_2 b_3\dots b_{l+3}\widehat{A}_{l+2}(z)= a_0 b_2 b_3\dots b_{l+2} b_{l+3}\widehat{A}_{l+1}(z)+\\+z a_0 a_{l+2} b_2 b_3\dots b_{l+1}\widehat{A}_{l}(z).
 \end{gathered}
 \end{equation} 
which gives
 \begin{equation}
  b_{l+2} b_{l+3}\widehat{A}_{l+2}(z)= b_{l+2} b_{l+3}\widehat{A}_{l+1}(z)+z a_{l+2}\widehat{A}_l(z).
 \end{equation}
Consequently, repeating the recurrence three times
 \begin{equation}\label{pierwszd_r}
 \widehat{A}_{l+1}=\widehat{A}_{l}(z)+zr_{l+1}\widehat{A}_{l-1}(z),\quad l\geq 0 .
 \end{equation} 
     \begin{equation}
     \widehat{A}_{l}(z)=\widehat{A}_{l-1}(z)+zr_{l}\widehat{A}_{l-2}(z) \label{drugid_r}
     \end{equation}
     \begin{equation}
     \widehat{A}_{l-1}(z)=\widehat{A}_{l-2}(z)+zr_{l-1}\widehat{A}_{l-3}(z) \label{trzecid_r}.
     \end{equation}  
      Substituting  $\widehat{A}_{l-2}(z)$ from equation (\ref{trzecid_r}) to equation (\ref{drugid_r}) and $\widehat{A}_{l}(z)$ from equation (\ref{drugid_r}) to equation (\ref{pierwszd_r}) we obtain \eqref{rekurencjaDlaA}.
  
 Analogously we derive the formula \eqref{rekurencjaDlaB} using \eqref{15} and \eqref{18} instead of \eqref{14} and \eqref{17}.
  
  \end{proof}
  
  We introduce now the key for stable numerical implementation: the numbers $h_l^k$. 
 	\begin{proposition}\label{ph}
Let $s_0,s_1\ldots\in\Comp$, $s_0,s_1\neq 0$ and let the numbers $h_l^k$ $(k\geq 1, l\geq -1)$
be given by a recurrence relation 
	 \begin{equation}\label{h}
	 \begin{gathered}
 	 h_{-1}^k=-\frac{s_k}{s_0},  \\
	h_0^k=\frac{s_k}{s_0}-\frac{s_{k+1}}{s_1}=-\left(h_{-1}^k+\frac{h_{-1}^{k+1}}{h_{-1}^1}\right),\\
 	  	 h_l^k=\frac{h_{l-2}^{k+1}}{h_{l-2}^1}-\frac{h_{l-1}^{k+1}}{h_{l-1}^1},\quad l\geq 1 .
		 \end{gathered}
 	  	 \end{equation}
	Then
	$r_l=h_{l-2}^1$ for $l\geq 0 $, with $r_l$ as in Lemma \ref{2rec}. Furthermore, $r_l$ $(l\geq 0)$ depends  only on the first $l+1$ entries of the sequence $s_0,s_1, \dots$. 
	\end{proposition}

\begin{proof}
Note that for $l\geq 1 $ we get by \eqref{999} that
 	 \begin{eqnarray*}
 	 r_l & = & (-1)^{l-1}\frac{f_l^0f_{l-3}^0}{f_{l-2}^0f_{l-1}^0}\\
 	 &=& (-1)^{l-1}\frac{f_{l-3}^0}{f_{l-2}^0f_{l-1}^0}\frac{(-1)^{l-1}}{f_{l-3}^0}[f_{l-1}^0f_{l-2}^1-f_{l-2}^0f_{l-1}^1]\\
 	 &=&\frac{f_{l-2}^1}{f_{l-2}^0}-\frac{f_{l-1}^1}{f_{l-1}^0}.
 	 \end{eqnarray*}
 We set
 	 \begin{equation}
 	 h_l^k:=g_l^k-g_{l+1}^k,\quad g_l^k=\frac{f_l^k}{f_l^0}, \quad k\geq 1  ,\ l\geq -1 ,
 	 \end{equation}
so that
 $$
 r_l=h_{l-2}^1,\quad l\geq 0 .
 $$
It remains to prove that  $h_l^k$ satisfy the recursive relations \eqref{h}. We first note that for $k\geq 0  $ we have
 \begin{eqnarray*} 
 	 h_{-1}^k&=&g_{-1}^k-g_0^k=-\frac{s_k}{s_0},\\
	 h_0^k&=&g_0^k-g_1^k=\frac{s_k}{s_0}-\frac{s_{k+1}}{s_1}=-\left(h_{-1}^k+\frac{h_{-1}^{k+1}}{h_{-1}^1}\right).
 \end{eqnarray*}
 Furthermore, for $k\geq 1 $,
 we get the recurrence 
 \begin{equation}
 	  	 h_l^k=g_l^k-g_{l+1}^k=\frac{g_{l-2}^{k+1}-g_{l-1}^{k+1}}{g_{l-2}^1-g_{l-1}^1}-\frac{g_{l-1}^{k+1}-g_l^{k+1}}{g_{l-1}^1-g_l^1}=\frac{h_{l-2}^{k+1}}{h_{l-2}^1}-\frac{h_{l-1}^{k+1}}{h_{l-1}^1},
 	  	 \end{equation}
	where the second equality results from the fact that by \eqref{999} we have for  $l\geq -1$
 	 \begin{eqnarray*}
 	 g_{l+1}^k &= & \frac{f_l^0f_{l-1}^{k+1}-f_{l-1}^0f_l^{k+1}}{f_l^0f_{l-1}^1-f_{l-1}^0f_l^1}\\ 
	&=&  \frac{f_l^0f_{l-1}^0\left[\frac{f_{l-1}^{k+1}}{f_{l-1}^0}-\frac{f_l^{k+1}}{f_l^0}\right]}{f_l^0f_{l-1}^0\left[\frac{f_{l-1}^1}{f_{l-1}^0}-\frac{f_l^1}{f_l^0}\right]}=
 	 \frac{g_{l-1}^{k+1}-g_l^{k+1}}{g_{l-1}^1-g_l^1}.
 	 \end{eqnarray*}	 
		
		To see the `furthermore' part note that in order to compute $r_l=h_{l-2}^1$ we need $h_{l-3}^k$, $k=1,2$. Proceeding further by induction w.r.t. $j=1,\dots l-1$ we see that $r_l$ depends only on  $h_{l-2-j}^k$,  $k=1,\dots, j+1$. The case  $j=l-1$ proves the statement.

 \end{proof}

\subsection{Change of variables and matrix formulation}\label{matrixform}
Recall that our final aim is approximation of the function 
   \begin{equation}
   Z(w)=\sum_{i=1}^\infty s_iw^{-i}=G\left(z\right),
   \end{equation}
   where $w=z^{-1}$. 
   Let us introduce the  monic polynomials ($n\geq 0$)
    \begin{equation}\label{PQdef}
 P_{n}(w)=w^{n-1}\widehat{A}_{2n-1}\left(w^{-1}\right), \quad
   Q_{n}(w)=w^n\widehat{B}_{2n-1}\left(w^{-1}\right),
   \end{equation}
   and 
   \begin{equation}\label{PQdef2}
  \widetilde P_{n}(w)=w^{n}\widehat{A}_{2n}\left(w^{-1}\right),\quad
  \widetilde Q_{n}(w)=w^n\widehat{B}_{2n}\left(w^{-1}\right).
   \end{equation}
  Note that with $w\to\infty$ one has by  \eqref{finalPades} that 
   $$
   \frac{s_0wP_n(w)}{Q_n(w)} - Z(w) = \mathcal{O}(w^{-2n}),\quad     \frac{s_0\widetilde{P}_n(w)}{\widetilde{Q}_n(w)} - Z(w) = \mathcal{O}(w^{-2n}) .
   $$
  
Based on  (\ref{rekurencjaDlaB}) we get the following four recurrences:     
   
  \begin{equation}\label{32a}
   \begin{gathered}
   P_{0}(w)=0, \qquad P_1(w)=1\\
   P_{l+1}(w)=(w-c_l)P_l(w)-d_lP_{l-1}(w),\quad l\geq 1 ,\\
      Q_{0}(w)=1, \qquad Q_1(w)=w+r_1\\
   Q_{l+1}(w)=(w-c_l)Q_l(w)-d_lQ_{l-1}(w),\quad l\geq 1 ,\\   
   \end{gathered}
   \end{equation}
where   \begin{equation}\label{777}
   \begin{gathered}
   c_l=-(r_{2l}+r_{2l+1}),\quad  d_l=r_{2l-1}r_{2l},\quad l\geq 1;
   \end{gathered}
   \end{equation}
and
\begin{equation}\label{32b}
   \begin{gathered}
\widetilde {P}_{0}(w)=1, \qquad\widetilde{ P}_1(w)=w+r_2\\
 \widetilde P_{l+1}(w)=(w-\tilde c_l)\widetilde P_l(w)-\tilde d_l\widetilde P_{l-1}(w),\quad l\geq 1 ,\\   
     \widetilde{ Q}_{0}(w)=1, \qquad\widetilde{ Q}_1(w)=w+r_1+r_2\\
     \widetilde Q_{l+1}(w)=(w-\tilde c_l)\widetilde Q_l(w)-\tilde d_l\widetilde Q_{l-1}(w),\quad l\geq 1 ,\\   
  \end{gathered}
   \end{equation}
where \label{777b}
  \begin{equation}
   \begin{gathered}
   \tilde c_l=-(r_{2l+2}+r_{2l+1}),\quad  \tilde d_l=r_{2l+1}r_{2l},\quad l\geq 1.\\
   \end{gathered}
   \end{equation}

 Here we arrive at the main point of the paper -- construction of the tridiagonal $J_n$ matrix and its submatrix $J'$:
 $$
  J_{n} = \begin{bmatrix} c_0 & 1 & \\
	 	 	 	d_1 & c_1 & 1  &  &  \\ 
	 	 	 	 & d_2 & \ddots & \ddots&   \\
	 	 	 	&  & \ddots & \ddots & 1& \\
	 	 	 	 & &  & d_{n-1} & c_{n-1}
	 	 	 	\end{bmatrix}, \quad 
				 J_{n-1}' = \begin{bmatrix} c_1 & 1 & \\
	 	 	 	d_2 & c_2 & 1  &  &  \\ 
	 	 	 	 & d_3 & \ddots & \ddots&   \\
	 	 	 	&  & \ddots & \ddots & 1& \\
	 	 	 	 & &  & d_{n-1} & c_{n-1}
	 	 	 	\end{bmatrix},
$$
where $c_0=-r_1$.
Analogously we define the two matrices $\widetilde{J}_n$ and $\widetilde{J}_n'$:
 $$
  \widetilde{J}_{n} = \begin{bmatrix} \tilde c_0 & 1 & \\
	 	 	 	\tilde d_1 & \tilde c_1 & 1  &  &  \\ 
	 	 	 	 & \tilde d_2 & \ddots & \ddots&   \\
	 	 	 	&  & \ddots & \ddots & 1& \\
	 	 	 	 & &  & \tilde d_{n-1} & \tilde c_{n-1}
	 	 	 	\end{bmatrix}, \quad 
				 \widetilde{J}_{n}' = \begin{bmatrix} \tilde c'_0 & 1 & \\
	 	 	 	\tilde d_1 & \tilde c_1 & 1  &  &  \\ 
	 	 	 	 & \tilde d_2 & \ddots & \ddots&   \\
	 	 	 	&  & \ddots & \ddots & 1& \\
	 	 	 	 & &  & \tilde d_{n-1} & \tilde c_{n-1}
	 	 	 	\end{bmatrix},
$$
where $\tilde c_0= -r_1-r_2$ and $\tilde c'_0= -r_2$. The following result expresses the zeros of $P_n(w)$, $Q_n(w)$, $\widetilde P_n(w)$, and $\widetilde Q_n(w)$ as eigenvalues of the above four matrices.
\begin{theorem}\label{PQdet}
For $n\geq 1$ we have 
$$
\begin{array}{lcr}
P_{n}(w)=\det(wI_{n-1} - J'_{n-1}),\quad & Q_n(w)=\det(wI_{n} - J_{n}),\\
 \widetilde{P}_{n}(w)=\det(wI_n - \widetilde{J}'_{n}),\quad &\widetilde{Q}_n(w)=\det(w I_{n} -\widetilde{J}_{n})
\end{array}
$$
where $I_{n}$ denotes the identity matrix of size $n$. Consequently,
 \begin{align*}
 \mathbf{e}_0^* \left(wI_n-J_{n}\right)^{-1}\mathbf{e}_0&=\frac{P_n(w)}{Q_n(w)}, \quad n\geq 1,\\
 \frac{ \mathbf{e}_0^* \left(wI_{n}-\widetilde{J}_{n}\right)^{-1}\mathbf{e}_0}{ \mathbf{e}_{0}^* \left(wI_{n
 }-\widetilde{J}_{n}'\right)^{-1}\mathbf{e}_{0}}&=\frac{\widetilde{P}_n(w)}{\widetilde{Q}_n(w)}, \quad n\geq 1,
 \end{align*}
where $\mathbf{e}_0$ is the first vector of the canonical basis.
\end{theorem}

\begin{proof}
	 	Let
	$$
	j_{n}(w):=\det(wI_{n+1}-J_{n+1}).
	$$
 Note that by the Laplace expansion  the following recurrence holds
	 \begin{equation}
j_n(w) = (w-c_{n-1})j_{n-1}-d_{n-1} j_{n-2},\quad n\geq 1,
 \end{equation} 
where the initial conditions are given by	 
	 $$
	 j_{-1}(w):=0,\quad j_0(w): = 1.
	 $$
Comparing this with \eqref{32b} gives the statement concerning $Q_n(w)$. The statements concerning $P_n(w)$, $\widetilde{P}_n(w)$, $\widetilde{Q}_n(w)$ can be shown in a similar way. The proof of the 'Consequently' part follows directly by the adjugate matrix formula for the inverse.
 
\end{proof}
A standard diagonalization for the $J_n$ matrix is the following. 
\begin{lemma}\label{diagonalization}
Let $n\geq 1$ and let $s_0,\dots,s_{2n-1}\in\Comp$ be such that the zeros $z_0^n,\dots z_{n-1}^n$ of  $Q_{n}(w)$ are simple. Then the matrix
$$
T_n=\begin{bmatrix}T_n[i,j]\end{bmatrix}_{ij=0}^{n-1}=\begin{bmatrix}Q_{i}(z_j^n)\end{bmatrix}_{ij=0}^{n-1}
$$ 
is the diagonalisation  of $J_{n}$, i.e.
 $$
 J_{n}=T_n \diag(z_0^n,\dots z_{n-1}^n) T_n^{-1}.
 $$
 Similarly, if the zeros $\tilde z_0^n,\dots \tilde z_{n-1}^n$ of $\widetilde Q_n(w)$ are simple then $\widetilde T_n=[\widetilde Q_{i}(z_j^n)]_{ij=0}^{n-1}$ is the diagonalisation    of $\widetilde J_{n}$.
\end{lemma}

\begin{proof}
Note that by \eqref{32a} we have
 $$
   (z_j^nI_{n}- J_{n})\begin{bmatrix}
	 	Q_0(z_j^n)\\
	 	\vdots \\
		Q_{n-2}(z_j^n)\\	 
	 	Q_{n-1}(z_j^n)\\	 	
	 	\end{bmatrix}=\begin{bmatrix}
	 	0\\
	 	\vdots \\
	 	0\\
		Q_{n}(z_j^n)	 	
	 	\end{bmatrix}=
	 	0,
$$	 
i.e. the vectors $[Q_0(z_j^n),\dots,Q_{n-1}(z_j^n)]^\top$ $(j=1,\dots n)$ constitute the eigenvector basis for $J_{n}$.
 The statement concerning $\widetilde J_n$ can be proved similarly.
\end{proof}

 Key point of the above lemma is the assumption that the zeros of $Q_n(w)$ are simple. We show now this is a generic property. More precisely, by a generic subset of $\Comp^k$ we mean a set, whose complement is a subset of a proper algebraic subset of $\Comp^k$. 
Due to the presence of noise any realisation of the signal lies almost surely, under mild assumptions, in a given generic set. For formulating the statement recall the following classical formulas (cf. \cite{jacobi1846}):  
\begin{align}
Q_{n}(w)& = \det\left[{\begin{array}{cccc} 
    s_n & s_{n-1} & \cdots & s_0 \\
    s_{n+1} & s_n & \cdots & s_1 \\
     \cdot & \cdot & \cdot & \cdot \\
     \cdot & \cdot & \cdot & \cdot \\
     \cdot & \cdot & \cdot & \cdot \\
    s_{2n-1} &   s_{2n-2} & \cdots & s_{n-1}  \label{JacQ}\\
    1 & w & \cdots & w^{n}
  \end{array}}\right],\\
\widetilde{Q}_{n}(w)& = \det\left[{\begin{array}{cccc} 
    s_{n+1} & s_{n} & \cdots & s_1 \\
    s_{n+2} & s_{n+1} & \cdots & s_2 \\
     \cdot & \cdot & \cdot & \cdot \\
     \cdot & \cdot & \cdot & \cdot \\
     \cdot & \cdot & \cdot & \cdot \\
    s_{2n} &   s_{2n-1} & \cdots & s_{n} \\
    1 & w & \cdots & w^{n}
  \end{array}}\right].\label{JacQt}
\end{align}
In particular, $\widetilde Q_n(w)$ depends on $s_1,\dots,s_{2n}$ only.

\begin{theorem}\label{simplezeros} 
Let $n\geq 1$. The zeros of the polynomials $Q_n(w)$  $($respectively, $\widetilde Q_n(w))$ are simple for all $\mathbf s=[s_0,\dots,s_{2n-1}]^\top$ $( \mathbf s=[s_1,\dots,s_{2n}]^\top )$ in some generic subset of $\Comp^{2n}$. In particular, if  $\mathbf s$ is a random vector that has a continuous distribution on $\Comp^{2n}$, then the zeros of $Q_n(w)$ $($of $\widetilde Q_n(w))$ are almost surely simple. 
\end{theorem}

\begin{proof}
We show only the part concerning $Q_n(w)$, the proof for $\widetilde Q_n(w)$ is similar.  Due to \eqref{JacQ} $Q_n(w)$ can be written as 
$$
Q_n(w)= \sum_{j=0}^n p_j(s_0,\dots,s_{2n-1}) w^j,
$$
where $p_j$ is some polynomial in multiple variables. The polynomial $Q_n(w)$ has a double zero if and only if the Sylvester resultant matrix  $S(Q_n,Q_n')$ is singular \cite{holtz2012structured,sylvester1840xxiii}. Note that 
$$
p(s_0,\dots,s_{2n-1}):=\det(S(Q_n,Q_n'))
$$
 is a polynomial in $s_0,\dots,s_{2n-1}$. Therefore, the set of all $s_0,\dots ,s_{2n-1}$ for which the zeros of $Q_n(z)$ are not simple  is an algebraic subset of $\Comp^{2n}$. To show that it is a proper algebraic subset we need to show that
  $p$ is a nonzero polynomial. Observe that  for $s_j=\delta_{j,0}+\delta_{j,n}$ ($j=0,\dots,2n-1$) we get
$Q_n(w)= w^n-1$, which has no double roots. In consequence  $p(s_0,\dots, s_{2n-1})\neq 0$. 
\end{proof}

\begin{remark}\label{Hankel}
Theorem \ref{PQdet} gives a direct way of computing the zeros of $Q_n(z)$. However, note that the zeros of $Q_n(z)$ can also be computed using different methods. In particular, it is known  that they are eigenvalues of the linear pencil
$$
U_0 - \lambda U_1,
$$
where $U_0,U_1\in\Comp^{n+1}$ are Hankel matrices defined by
$$
U_0[i,j] = s_{i+j+1},\quad U_1[i,j]=s_{i+j},\quad i,j=0,\dots, n,
$$
as $\det(U_0-zU_1)=Q_n(z)$, see \cite{barone2008new}.
\end{remark}

We postpone the comparison of the numerical results of the two methods to Section \ref{comp2} and conclude the present section by  formulating a purely linear algebraic corollary, announced already in the Introduction.

\begin{corollary}\label{cor:linalg}
Consider the Hankel pencil $U_0-zU_1$ defined above. Then for all $s_0,\dots,s_{2n-1}\in\Comp$ except a proper semialgebraic subset of $\Comp^{2n}$ its eigenvalues are simple and in such case
the matrix $J_n$ constructed above is a companion matrix of the pencil, in the sense that the spectra of $J_n$ and $U_0-zU_1$ coincide. 

\end{corollary}

\section{ The  residues}\label{residues}
In the last two sections of the paper we shall  concentrate on the $[n-1/n]$ Pad\'e  approximation.
And so let $z^n_i$ ($i= 0,\dots,n-1$) 
denote the zeros of $Q_n(w)$ and $\lambda^n_i$ ($i=1,\dots,n-1$) the zeros of $P_n(w)$.
 Our aim is to find a numerically effective method to calculate the residues of the rational function 
 $$
 Z_{[n-1/n]}(w)={s_0wP_{n}\left(w\right)}/{Q_{n}\left(w\right)}.
 $$
 Note that the leading terms of $P_n(w)$ and $Q_n(w)$ equal one, due to \eqref{PQdef} and that $\hat A_{2n-1}(0)=\hat B_{2n-1}(0)=1$. Consequently,  
\begin{equation}\label{Zn}
 	 	\frac{s_0wP_{n}\left(w\right)}{Q_{n}\left(w\right)}=s_0w\frac{\prod_{i=1}^{n-1}(w-\lambda^n_i)}{\prod_{i=1}^{n}(w-z^n_i)}.
\end{equation}
 Different formulas are available; here we list six of them. The first three ones immediately stem from the sections above:

 \begin{theorem}\label{resid}
Let $n\geq 1$ and assume that $Q_n(w)$ has simple zeros only. Then the residues $\rho_j^n$ of $Z_{[n-1/n]}(w)$ corresponding to the poles $z_j^n$ $(j=0,\dots,n-1)$ are given by
the following three formulas:
\begin{equation}\label{rho1}
		\rho_j^n=s_0z_j^n\frac{\prod_{i=0}^{n-2}(z_j^n-\lambda^n_i)}{\prod_{\substack{i=0\\ i\ne j}}^{n-1}(z_j^n-z_i^n)};
 	 	\end{equation}
\begin{equation}\label{rho3}
\rho_j^n=s_0z_j^n \frac{P_n(z_j^n)}{Q_n'(z_j^n)};
 \end{equation}
\begin{equation}\label{rho2}
 	 	\rho_j^n= s_0 z_j^n T_n^{-1}[j,0];
 \end{equation}
where 
$T_n$  is the diagonalization  of $J_{n}$ as in Lemma \ref{diagonalization}.
 \end{theorem}

 \begin{proof} Equations \eqref{rho1} and \eqref{rho3} follow directly from the definition of the residue of a rational function. To see equation \eqref{rho2} note that 
\begin{eqnarray*}
 \mathbf{e}_0^* \left(wI_n-J_{n}\right)^{-1}\mathbf{e}_0 &=& \mathbf{e}_0^* T_n \diag( w- z_0^n,\dots w-z_{n-1}^n)^{-1} T_n^{-1} \mathbf{e}_0\\
 & =& \sum_{j=0}^{n-1}  \mathbf e_0^* T_n \mathbf e_j\mathbf e_j^* T_n^{-1} \mathbf e_0(w-z_j^n)^{-1}
\end{eqnarray*}
and it is enough to use Theorem \ref{PQdet}  and note that by definition 
$$
\mathbf{e}_0^* T_n \mathbf e_j=T_n[0,j]=Q_0(z^n_j)=1.
 $$
 \end{proof}

 \begin{remark}
  As the factors in  \eqref{rho1} can be very small, in its numerical implementation we first calculated $\log\rho_j^n$ as an appropriate sum. The values $Q_l'(z_j^n)$ in \eqref{rho3} can be calculated recursively by differentiating \eqref{32b}:
 \begin{equation}\label{32b'}
   \begin{gathered}
   Q'_{0}(z_j^n)=0, \qquad Q'_1(z_j^n)=1\\
   Q'_{l+1}(z_j^n)=(z_j^n-c_l)Q'_l(z_j^n)+Q_l(z_j^n)-d_lQ'_{l-1}(z_j^n),\quad l\geq 1. \\   
   \end{gathered}
   \end{equation}
 \end{remark}
 
A fourth method to calculate the residues of a Pad\'e approximation --by far the most accurate numerically, as we shall see in Subsection \ref{comp2}-- is based on the  following proposition.

\begin{proposition}\label{vanresid}
Let $\mathbf{s}=[s_0,\dots,s_{2n-1}]^\top$ be the vector of the given signal $s$ and let $V\in\Comp^{2n\times n}$ be the matrix with elements $V[k,j]= (z^n_j)^{k-1}$ $(k=0,\dots 2n-1$,  $j=0,\dots,n-1)$. Then the overdetermined  consistent system of linear equations 
\begin{equation}\label{van}
\mathbf{s}=V{\mathbf x}
\end{equation}
has precisely one solution $\mathbf x={\boldsymbol\rho}=[\rho_0^n,\dots,\rho^n_{n-1}]^\top$. 
\end{proposition}

\begin{proof}
From \eqref{Zn} we have with $z=w^{-1}$ that 
\begin{equation}
Z_{[n-1/n]}(w)=s_0w \sum_{j=0}^{n-1} \frac{\rho^n_j/(s_0z_j^n)}{w-z_j^n}=
\sum_{j=0}^{n-1} \frac{\rho_j^n/z^n_j}{1 -  z\ z^n_j},
\end{equation}
which after expansion into  a  Taylor series at $z=0$   gives
\begin{equation}\label{rico}
Z_{[n-1/n]}(z)
=\sum_{k=0}^\infty \sum_{j=0}^{n-1} \rho_j^n (z^n_j)^{k-1}z^k.
\end{equation}
As $Z_{[n-1/n]}(z)$ is the $[n-1/n]$ Pad\'e approximation of the function given by  Eq. (\ref{szere}), we have
\begin{equation}\label{serie}
s_k= \sum_{j=1}^{n} \rho_j^n (z^n_j)^{k-1} \qquad k=0,\cdots,2n-1,
\end{equation}
which is exactly stating that ${\boldsymbol \rho}$ is a solution of \eqref{van}. By uniqueness of the $[n-1/n]$ Pad\'e approximation, the solution is unique.
\end{proof}

\begin{remark} We explain here some details for the numerical implementation of Proposition \ref{vanresid} above. First, we observe that extra care should be taken in the construction of the matrix $V$: especially when $n$ is large the matrix $V$ could be ill-conditioned or even not be possible to obtain due to numerical overflow. 
We therefore calculate not the  matrix $V$ but the following product:
$$
\tilde V= V D,\quad  D=\diag(d_0,\dots,d_{n-1}),\quad d_j=\begin{cases} 1 :& |z_j^n|\leq 1\\ 
|z_j^n|^{-2n+2} : & |z_j^n|>1 \end{cases}.
$$
The columns of the matrix $\tilde V$ should be built recursively. Namely, if $|z^n_j|>1$,  we start from the last entry $\tilde V[2n-1,j]=1$ and construct the column backward: $\tilde V[k-1,j]=\tilde V[k,j]/z_j^n$. The columns $j$ for which $|z^n_j|\leq 1$ are constructed directly starting from $\tilde V[0,j]=(z^n_j)^{-1}$.

Second, we note  that $\tilde{\boldsymbol\rho}:=D^{-1}\boldsymbol\rho$  is the unique solution of  the overdetermined  consistent system 
\begin{equation}\label{tildevan}
 \tilde V \tilde{\boldsymbol \rho}= {\mathbf{s}}.
 \end{equation}
However,  due to the roundoff errors appearing when calculating the poles $z_j^n$, the above system  is not consistent in practice.
 Therefore, to calculate $\tilde{\boldsymbol\rho}$  we perform a
least square evaluation of the residues by solving the normal system:
\begin{equation}
\tilde V^\top\mathbf{s}=(\tilde V^\top \tilde V)\tilde{\boldsymbol\rho}
\end{equation}
which can be easily done using the  \verb-Matlab- operator   \verb+\+ (\verb-mldivide-): $\tilde {\boldsymbol\rho}=\tilde V\backslash \mathbf{s}$. As the last step of the procedure we multiply $ {\boldsymbol\rho}=D\tilde{\boldsymbol{\rho}}$. Note that some of the residues obtained this way are hard zeros, see Figure \ref{comp}.

For sake of completeness, we mention that the residues may be also computed as a solution of the system
\begin{equation}\label{van/2}
V [k=0:n-1,j=0:n-1] \boldsymbol\rho=  [s_0,\dots,s_{n-1}]^\top,
\end{equation}
(this formula was proposed in \cite{barone2010}), or by solving any other square subsystem  of the system \eqref{van}. 
As we shall see in the next section, this method, while faster, gives slightly worse results than solving the full system \eqref{van}.
\end{remark}

 \begin{remark}  
In the $[n/n]$ case, we can again use eq. \eqref{van} to calculate the residues, but with $\mathbf{s}=[s_1,\dots,s_{2n}]^\top$ and $V[k,j]= (z^n_j)^{k-1}$ $(k=1,\dots 2n$,  $j=0,\dots,n-1)$.
\end{remark}

 The last method  of computing the residues we list, important more from a theoretical point of view than a numerical one, is given by the following theorem. 
 
\begin{theorem}\label{pert}
Let  denote $\lambda_j^n(\tau)$ the eigenvalues of 
the matrix $J_{n}+\tau\mathbf{e}_0\mathbf{e}_0^*$, defined in such way that they are continuous functions of the parameter $\tau\in\Real^+$. Then for $\tau \to  \infty$ and
 for all $s_0,\dots,s_{2n-1}$ except a set of Lebesgue measure zero on which the eigenvalues of $J_{n}$ are not simple $($see Theorem \ref{simplezeros}$)$
\begin{enumerate}
\item[(i)] $n-1$ of the curves $\lambda_j^n(\tau)$ converge to the zeros $\lambda_1^n,\dots,\lambda_{n-1}^n$ of $P_n$ and one converges to infinity as $\tau+o(\tau)$;
\item[(ii)] we have
\begin{equation}\label{rho4}
\frac{d\lambda_j^n}{d\tau}(0)= \rho_j^n s_0z_j^n.
\end{equation}
\end{enumerate} 
\end{theorem}

\begin{proof}
The proof of statement (i) is  a standard reasoning based on the Rouche theorem, see e.g. \cite[Theorem 4.1]{RanW12}.  
Statement (ii) follows from the fact that $\lambda_j^n(\tau)$ ($\tau\geq 0$) satisfies the equation
$$
\left(\mathbf{e}_0^* (\lambda_j^n(\tau) I_n - J_{n})^{-1} \mathbf{e}_0 \right)^{-1}= \tau,
$$
which together with Theorem \ref{PQdet} gives
$$
\left(\sum_{j=0}^{n-1} \frac{\rho^n_j}{\lambda_j^n(\tau) -z^n_j}\right)^{-1}= s_0\lambda_j^n(\tau)\tau.
$$
 Differentiating both sides at $\tau=0$ we obtain  
$$
 \frac{(\lambda_j^n){}'(0)}{\rho^n_j}=  s_0 z_j^n .
$$
\end{proof}

 \begin{remark}
Although it is not stated above, one may easily show using the results  of \cite{RanW12} that the curves $\lambda_j(\tau)$ are analytic and do not intersect each other.  Therefore, Theorem \ref{pert} provides a method of joining the zeros $z_j^n$ of $Q_n(w)$ with the zeros $\lambda_j^n$ of $P_n(w)$ (an one with infinity) with analytic curves. A similar construction for the $[n/n]$ case can be obtained with the use of matrices $\widetilde J_n$ and $\widetilde J_n'$ of Theorem \ref{PQdet}. The relation of the nature these curves to the Froissart doublets is under investigation. 
  \end{remark} 

\section{Numerical tests}\label{num}

The total numerical error is a combination of four errors that appear during the procedure:
\begin{itemize}
\item calculation of the $J_n$ matrix in Section \ref{matrixform}. 
\item calculation of the eigenvalues of the $J_n$ matrix.
\item calculation of the corresponding residues.
\end{itemize}

 
The total backward or forward error is here very hard to estimate. Furthermore, one can easily show artificial  examples where each of the two algorithms for computing the poles  fails.  E.g., when $s_0$ or $s_1$  are numerical zeros, the construction of the $J_n$ matrix fails, while the pencil constructed as in  Remark \ref{Hankel} works correctly for $s_0=0$. Furthermore, the calculation of the poles via the pencil method  can fail for example in the case when the signal consists of a few (much less than one quarter of the number of data points) damped oscillations and very low noise, as in that case the pencil $U_0-\lambda U_1$ is too close to a singular pencil.   

In the current section we present a strategy for verifying the reliability of the numerical computations using the different formulas presented above. The tests provide a numerical computation of the forward and backward error. 

We here test the algorithms on data series $\mathbf{s}^n\in\Comp^{2n}$ consisting of complex Gaussian white noise $\mathbf s^n=\mathbf s^n_{1}+\text{i}\mathbf s^n_2$, where $\mathbf s^n_{1},\mathbf s^n_{2}\in\Comp^{2n}$ are normally distributed real vectors, with mean zero and variance one, generated --for each $n$ separately--by the  Matlab  {\texttt randn} function.
The reasons are threefold:
firstly, this allows us to easily identify the formulas for calculating poles and residues that in usual circumstances are unsuitable for numerical computations. Secondly, our computations for white noise can serve as a reference point in further studies on more complicated data sequences. 
Thirdly, the statistical properties of poles and zeros of the Pad\'e approximant of a pure noise data series are of interest both from the theoretical point of view \cite{barone2013universality, bessis2009universal, bessis2013noise} and from that of their of applications in signal detection \cite{perotti2012IEEE} and hence, reliable numerical methods are required for this particular case.

Our aim here is to study the stability and accuracy of the proposed methods as a function of the number of data points, by plotting the numerical error vs. $n$, which is half of the length of the above random signal. Note that  increasing the number of data points in most cases 
 increases the information on one hand, but on the other hand it also increases the numerical error. Hence, finding the proper length of the signal is always a matter of a careful study, the tools for which are provided below.

%
%
%

\subsection{Forward error of the  methods for computing the poles}\label{polesforward}

First, let us compare the computation time for the poles using the $J_n$ matrix method (Theorem \ref{PQdet}) and the pencil method (Remark \ref{Hankel}), as a function of the number of data points $2n$. The results are shown in Figure \ref{f2}: for $n=10^3$ the $J_n$ matrix method appears to be $10$ times faster, which for real time analysis might make the difference. 
The results for data sequences of various signals plus noise show the same behavior.
 The \texttt{Matlab} function \texttt{eig}  was used to compute both the eigenvalues of $J_n$ and of the linear pencil $U_0-\lambda U_1$. For the $J_n$ matrix, an ad-hoc LR routine, with origin shift to speed up convergence, would have the advantage of working with only two vectors instead than with a whole matrix \cite{parlett}.
\begin{figure}[htbp]
\includegraphics[width=300pt]{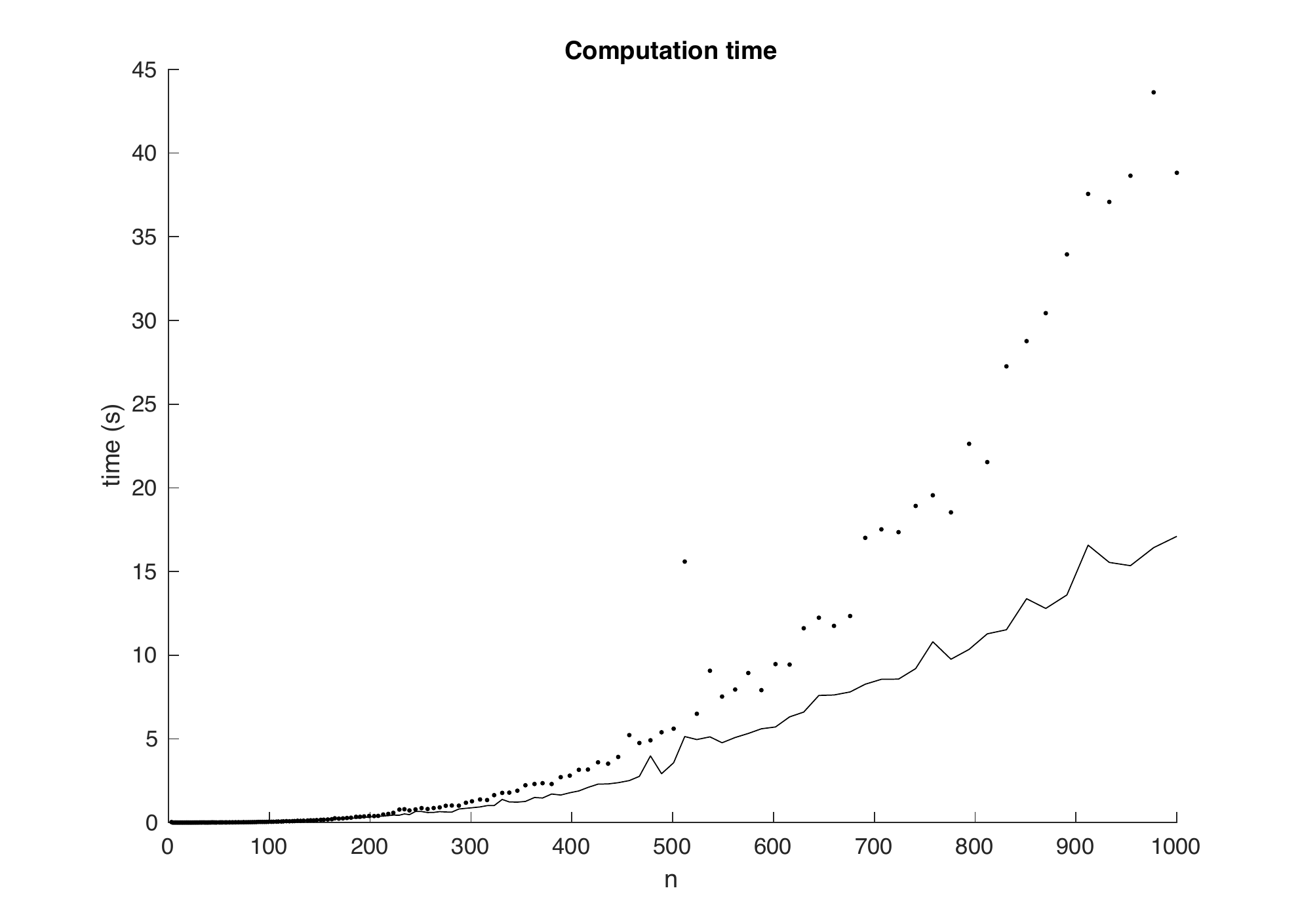}
\caption{ Computation time of the poles $\mathbf{z}^n$ for Gaussian white noise in the complex plane using the  $J_n$ matrix (solid line) and  the $U_0-\lambda U_1$ pencil (dots). }\label{f2}
\end{figure} 

We now pass to comparing the accuracy of the two methods of computing the eigenvalues by the following procedure: given a set of poles $\mathbf{z}^n=[z_0^n,\dots,z_{n-1}^n]^\top$ and the residues $\boldsymbol{\rho}^n=[\rho_0^n,\dots,\rho_{n-1}^n]^\top$, we use equation \eqref{van} to create a data sequence $\mathbf s$. 
For this $\mathbf s$ we calculate the poles $\widetilde{\mathbf z}$. Both algorithms have the property that $\widetilde{\mathbf z}$ equals $\mathbf z$  
in theory. We  set $\norm{\tilde{\mathbf z} ^n- {\mathbf z}^n }/\norm{\mathbf z^n}$ as a measure of the relative forward error in computing the poles and we calculate  it for both methods.

The results of simulations for consecutive values of $n$ can be found in Figure \ref{f2b}.  To avoid unrealistic values of poles and residues $\mathbf{z}^n$ and $\boldsymbol{\rho}^n$, we generate them from a complex Gaussian white noise sample $\hat{\mathbf s}$ using the pencil method and equation \eqref{van}, respectively. While both methods give quite small errors, a comparison of Figure \ref{f2b} with Figure \ref{f2} indicates that the advantage in speed of the $J_n$  matrix method is compensated by a loss of precision of approximately the same order of magnitude.  We note here in passing that we have observed cases of data series $\mathbf{s}$ comprising noise plus a signal of the form \eqref{oscilform}, where using the $J_n$ matrix gives a smaller error. These cases will be dealt with in a forthcoming paper.

\begin{figure}[htbp]
\includegraphics[width=300pt]{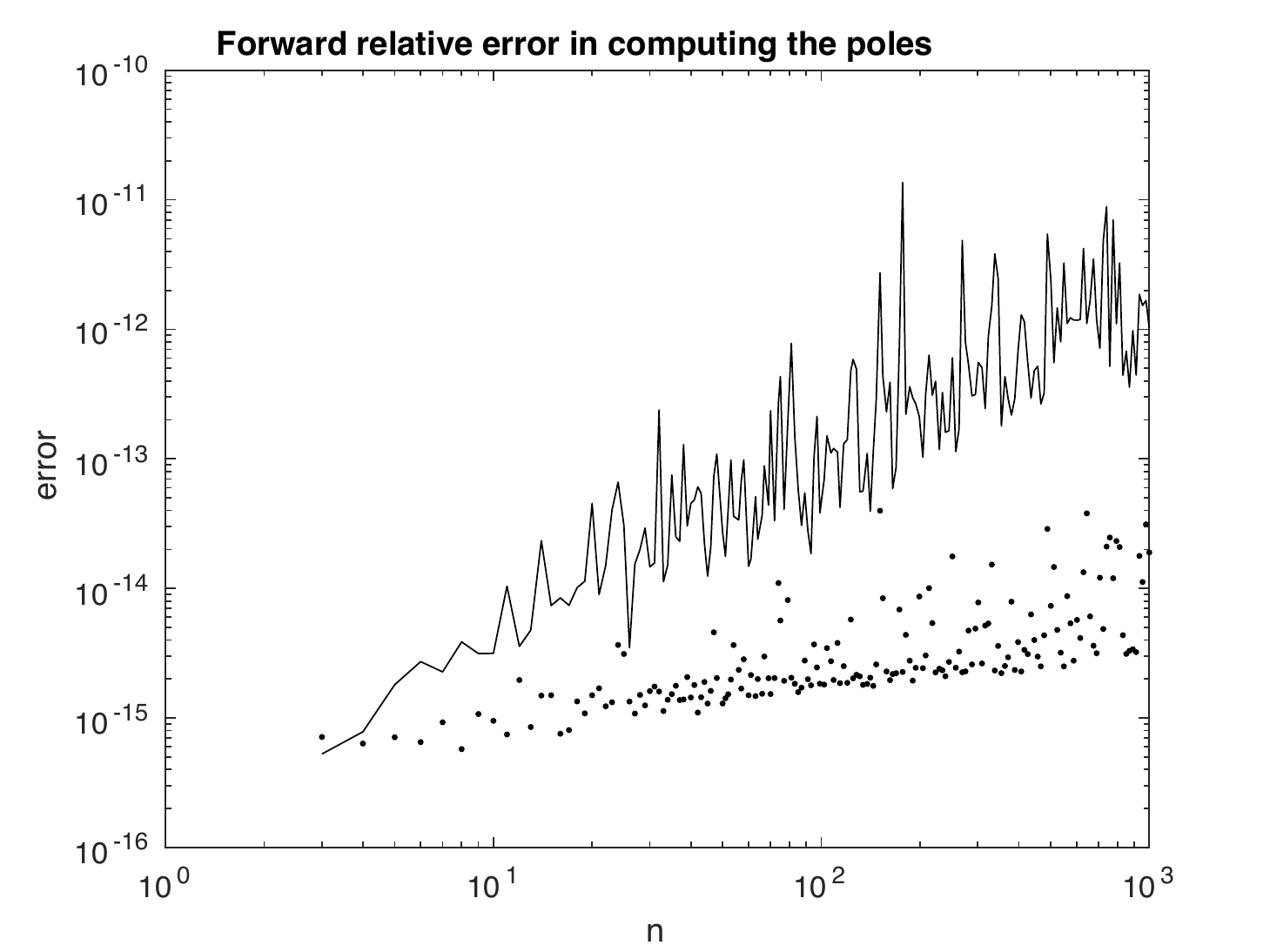}
\caption{ Forward relative error of the poles $\mathbf{z}^n$ for complex Gaussian white noise, computed using the  $J_n$ matrix (solid line) and  the $U_0-\lambda U_1$ pencil (dots).}\label{f2b}
\end{figure} 

\subsection{Forward error of the methods for computing the residues}\label{residsforward}

We adopt the same strategy to compute the forward error for the different methods to calculate the residues: the distance between the output $\widetilde{\boldsymbol \rho}^n$ of the chosen algorithm and the original residues ${\boldsymbol \rho}^n$, measured as $\norm{\widetilde{\boldsymbol \rho} ^n- {\boldsymbol \rho}^n }/\norm{\boldsymbol \rho^n}$ is a measure of the relative forward error in computing the residues. The results can be found in Figure \ref{official3}.
As only the $J_n$ method gives the zeros of the Pad\'e approximant needed for eq. \eqref{rho1} (and eq. \eqref{rho3} whose results we do not present here as they always are essentially identical to those of eq. \eqref{rho1}), the eigenvalues were computed with the $J_n$ method. Apart from the derivative formula  \eqref{rho4} all the methods appear to give comparable relative errors.
\begin{figure}[htb]
\includegraphics[width=300pt]{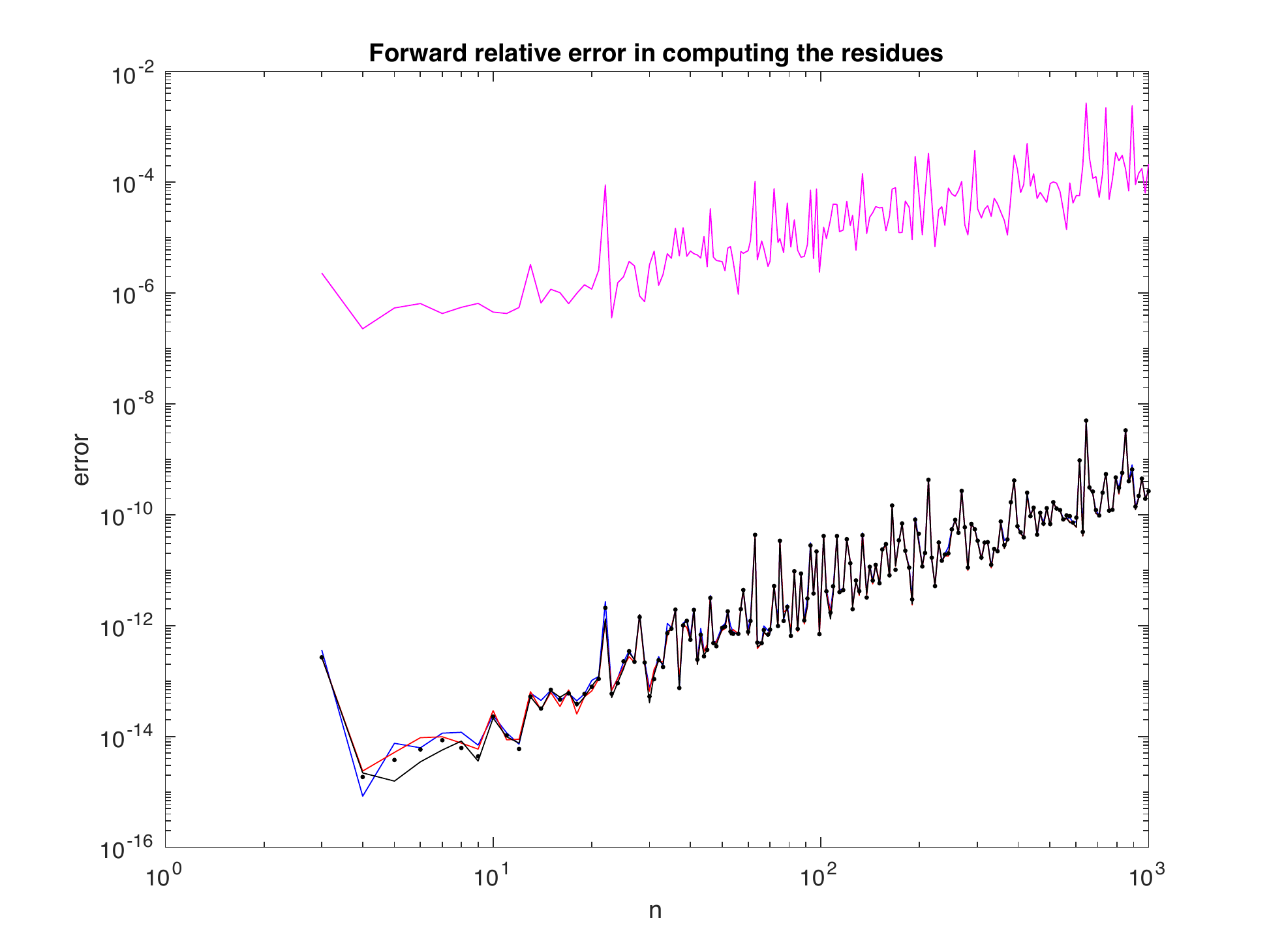}
\caption{ Forward relative error of the residues for complex Gaussian white noise, computed using the product \eqref{rho1} (red), eigenvectors of the $J_n$ matrix \eqref{rho2} (black), overdetermined Vandermonde system  \eqref{van} (blue),  square Vandermonde system \eqref{van/2} (black dots) and the derivative formula  \eqref{rho4} (magenta).}\label{official3}
\end{figure} 

\subsection{A global test for correctness of the residues computation}\label{global}

A fast test of the numerical reliability of equations \eqref{rho1},\eqref{rho3},\eqref{rho2},\eqref{van} and (\ref{rho4}) 
is based on the equality
 $$
 \sum_{j=0}^{n-1} \frac{\rho_j^n}{s_0z_j^n}=1
 $$ 
 that follows from the fact that by the Euler-Jacobi vanishing condition the sum of the residues of the fraction of two monic polynomials
\begin{equation}
\frac{\prod_{i=1}^{n-1}(w-\lambda^n_i)}{\prod_{i=0}^{n-1}(w-z^n_i)}.
\end{equation}
 equals one. The results of simulation can be found in Figure \ref{official4} and are essentially in agreement with those of Figure \ref{official3}, the Vandermonde methods \eqref{van} and \eqref{van/2} giving only slightly worse results than the  product \eqref{rho1} and $J_n$ matrix eigenvectors \eqref{rho2} methods. The eigenvalues were computed with the $J_n$ method.

 \begin{figure}[htbp]
\includegraphics[width=300pt]{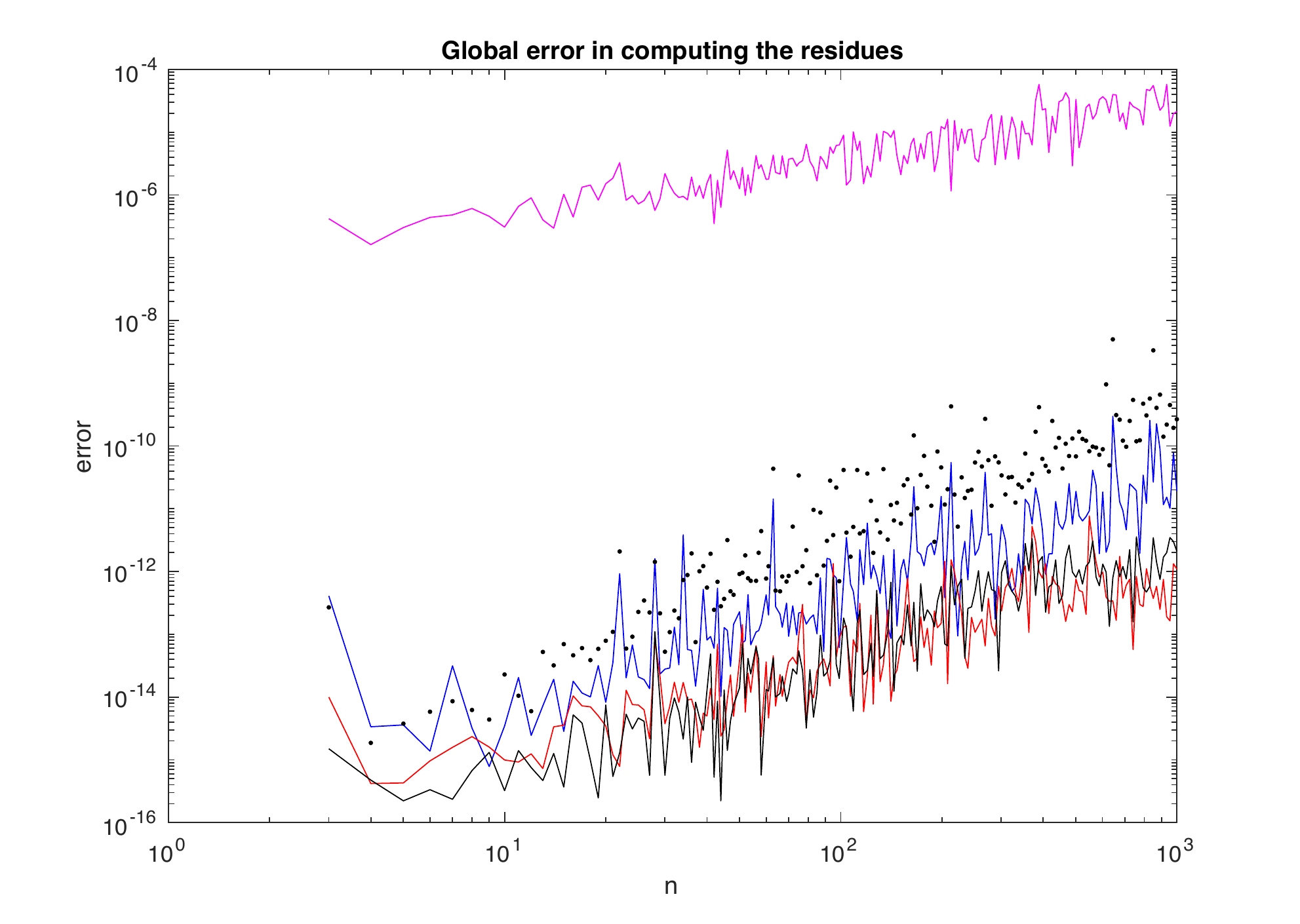}
\caption{Global error of the residues for complex Gaussian white noise, computed using the product \eqref{rho1} (red), eigenvectors of the $J_n$ matrix \eqref{rho2} (black), overdetermined Vandermonde system  \eqref{van} (blue),  square Vandermonde system \eqref{van/2} (black dots) and the derivative formula  \eqref{rho4} (magenta).}\label{official4}
\end{figure} 
 \begin{remark}
 In the $[n/n]$ case the corresponding Euler-Jacobi vanishing condition reads
  $$
 \sum_{j=0}^{n-1} \frac{\rho_j^n}{s_0}=\sum_{i=0}^{n-1}(z^n_i-\lambda^n_i).
 $$
 \end{remark}
  
\subsection{A graphical test for correctness of the residues computation}\label{comp2}
 
 In Figure \ref{comp} we plot the absolute value of the residues $|\rho_j^n|$ against the  poles' radial positions $|z_j^n|$  for complex Gaussian white noise. The residues are computed using five of the above methods (again, we did not plot the results for eq. \eqref{rho3} as they are very similar to those for eq. \eqref{rho1}).  The eigenvalues were computed with the $J_n$ method. For $|z_j^n|  \ge 1$ --where the magnitude of the residues sharply decreases--  the differences between  the Vandermonde methods and the other methods that give comparable results for the forward and global error become evident. The flattening of the magnitude of the residues on the right hand side depends on the method used and is clearly an artifact: (relatively) large residues for large  $|z_j^n|$'s cause the signal reconstructed using eq. \eqref{serie} to explode exponentially. More in detail: the amplitude of an addendum of the signal \eqref{serie} associated with a pole $z_j^n$ equals approximately $\sigma|\rho_j^n||z_j^n|^{(k-1)}$, where $\sigma$ is the noise amplitude. If $|z_j^n|<1$, the signal decays exponentially and the corresponding residue mostly affects the first terms of the series $s_k$; if instead $|z_j^n|>1$, the corresponding residue mostly affects the the last terms of ${\mathbf s}$. 
 The amplitude of an addendum of the signal associated with a pole $z_j^n$ outside the unit disc should therefore be bounded, i.e.  
 \begin{equation}\label{bdc}
 |\rho_j^n||z_j^n|^{2(n-1)}\lesssim \sigma.
 \end{equation}
  Writing  $|z_j^n|=1+\delta_j$ with $\delta_j \ll1$ we can approximate \eqref{bdc} by
   $$
   \ln|\rho_j^n|\lesssim  -2(n-1)\delta_j+\const,
   $$
    meaning that we expect the residues of poles outside the unit circle to decay exponentially with the pole's distance from the unit circle.

 From Figure \ref{comp} it clearly  appears that eq. (\ref{van}) --which follows the fall of the residues for $324$ orders of magnitude to the \texttt{Matlab} hard zero-- is the most effective numerically, while eq. (\ref{rho4}) gives the worst result. Eq. \eqref{van/2} gives results somehow similar to those of eq. \eqref{van}, but while the residues obtained using eq. \eqref{van} lie on a straight line coinciding with the bound \eqref{bdc}, the line formed by the residues calculated using eq. \eqref{van/2}  is (slightly) curved upward.
\begin{figure}[htb]
\includegraphics[width=300pt]{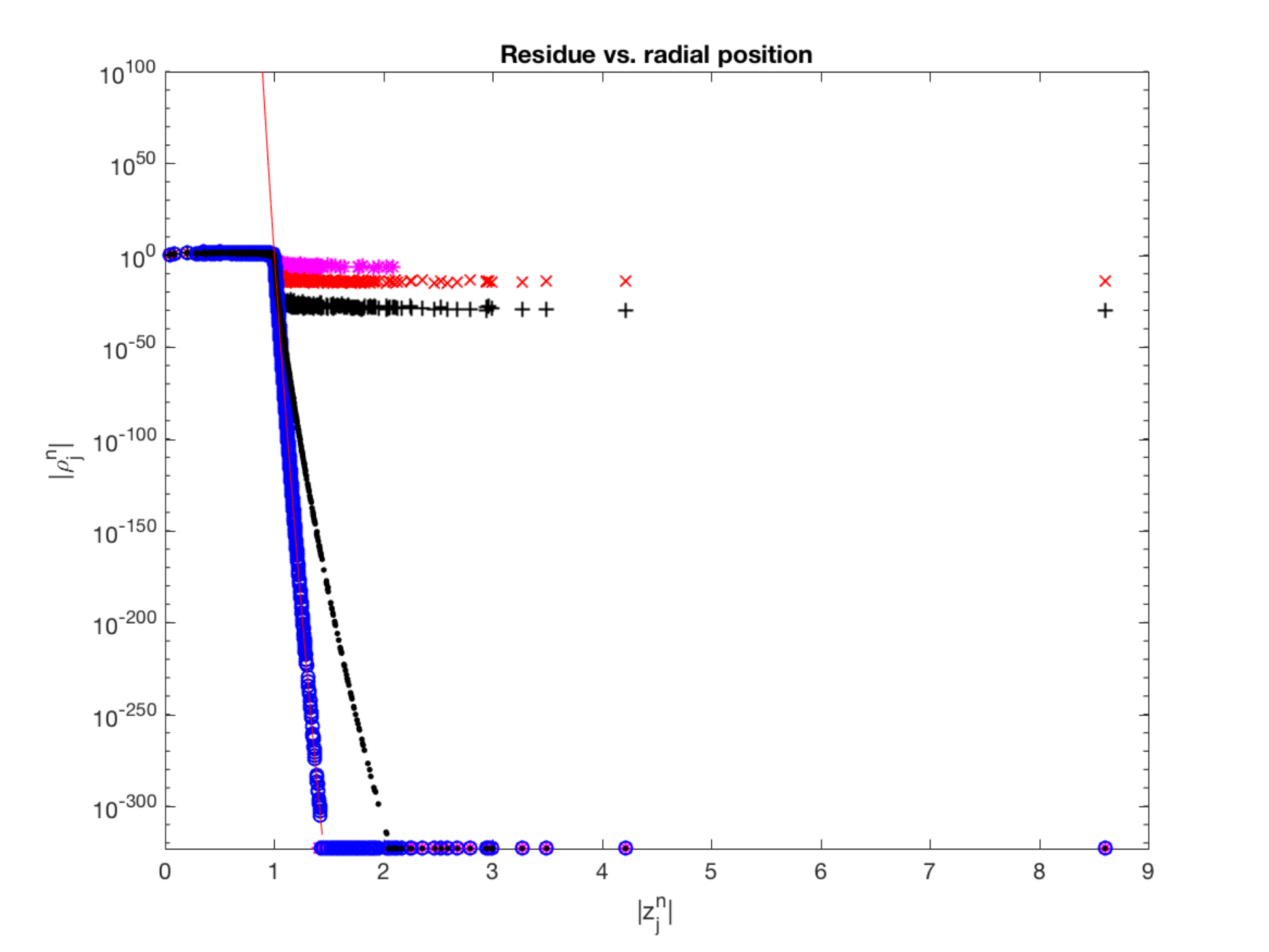}
\caption{Comparison of the magnitude of the residues for Gaussian white noise vs. radial position for $n=1000$, using \eqref{rho1} (red $\times$), \eqref{rho2} (black $+$), \eqref{van} (blue $\circ$), \eqref{van/2} (black dots) and \eqref{rho4} (magenta $*$); $200$ different realizations of the data sequence $s$ were used.  To guide the eye, eq. \eqref{bdc} is plotted as a red line. The bottom horizontal line of the residues corresponds to  $|\rho_j^n|=10^{-324}$, which is a hard zero in \texttt{Matlab}.}\label{comp}
\end{figure}

\subsection{Computing backward error by reconstructing of signal}\label{reco}

A comparison of the original data sequence $\mathbf s$ with the one reconstructed from the calculated poles and their residues by formula \eqref{serie}, denoted by $\widetilde{\mathbf s}$, can also serve as a test of the reliability of the procedure. From a numerical perspective, $\norm{\mathbf s-\widetilde{\mathbf{ s}}}/\norm{\mathbf s}$ is a measure of the backward error in computing the eigenvalues and poles. 

 The results for the four methods given by equations \eqref{rho1}, \eqref{rho2}, \eqref{van}, and \eqref{van/2} applied to complex Gaussian white noise are compared in Figure \ref{confr}. In all cases the signal was reconstructed using equation \eqref{tildevan}. The \verb-Matlab- function \verb-eig- was again used to compute poles as eigenvalues of the 
$J_n$ matrix. 
Only eq. \eqref{van} gives a good reconstruction (the error goes approximately as $3\cdot10^{-17}n^{3/2}$). Due to the explosion of the reconstruction discussed in Section \ref{comp2}, other methods of computing the residues not only give for $n=1000$ backward errors of order exceeding $10^{50}$, but for some noise realizations the reconstruction of the signal fails completely. This indicates a definite advantage of the \eqref{van} method.

\begin{figure}[htb]
\includegraphics[width=300pt]{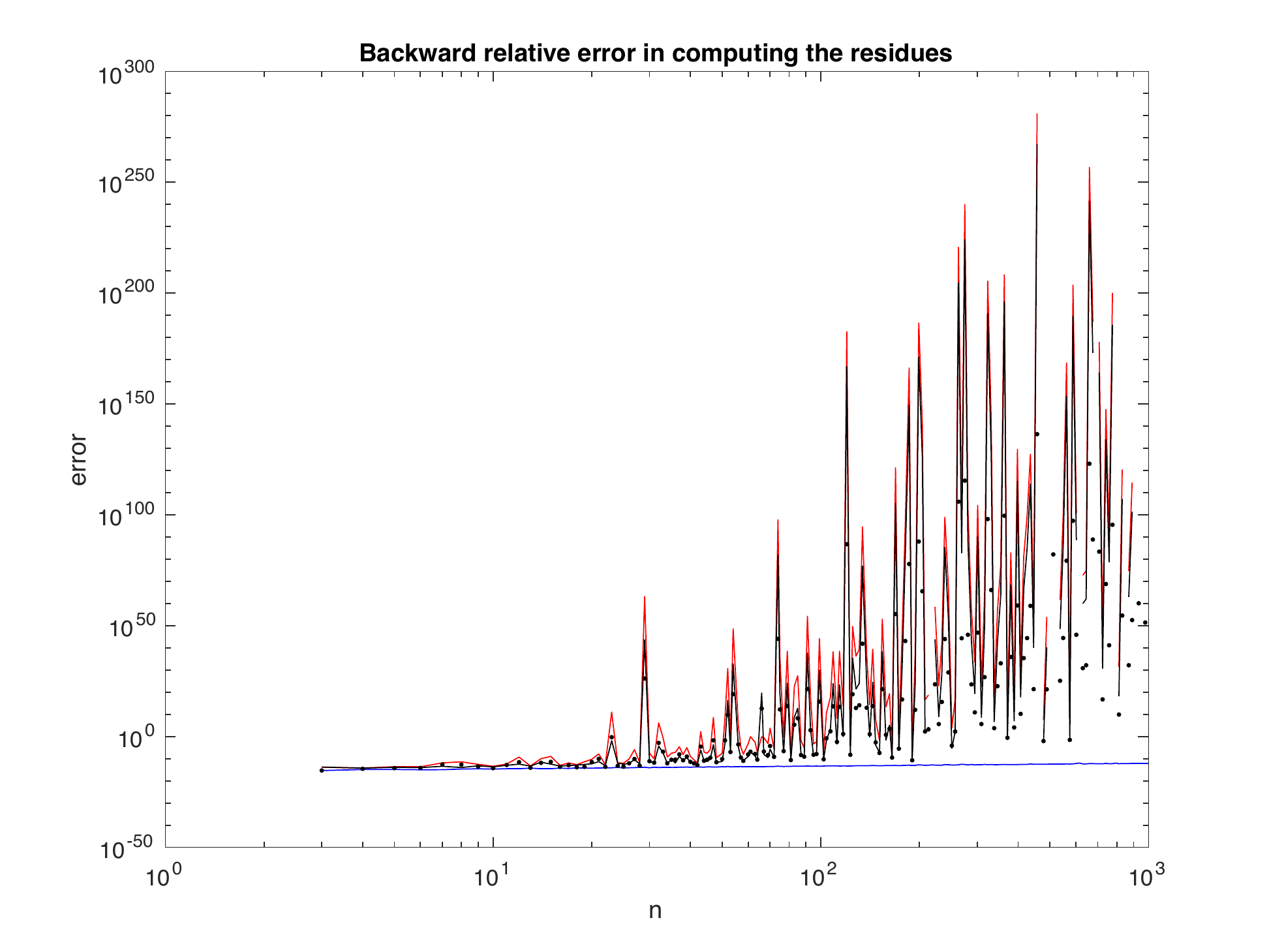}
\caption{Backward relative error of the whole procedure for complex Gaussian white noise, computed using the product \eqref{rho1} (red), the eigenvectors of the $J_n$ matrix \eqref{rho2} (black), the overdetermined Vandermonde system  \eqref{van} (blue), and the square Vandermonde system \eqref{van/2} (black dots).}\label{confr}
\end{figure} 
\section{ Conclusions}

We have presented two numerical method to calculate the poles  of Pad\'e approximations of data series and  six methods to calculate the residues. Using them, we have been able to calculate on a common desktop PC poles and zeros of noisy series up to $2\times 10^4$ data points. 
We also provided tools for checking the numerical stability of the algorithms on the given data series.
 The choice between the methods to calculate the poles appears to be not obvious, and in each case the backward and forward error should be compared.
A comparison of the different methods to calculate the residues of the poles indicates instead that eq. \eqref{van} is the  best choice for numerical implementation. 
 In the case the zeros of the Pad\'e approximation are sought instead or beside the residues of the poles, they are easily accessible by the $J_n$ matrix method.

 \section*{Acknowledgments} We would like to thank Christian Schr\"oder for many fruitful discussions and an essential help in implementing some of the numerical algorithms.

\bibliographystyle{abbrv}

\bibliography{articolo}

\end{document}